\documentclass[12pt, letterpaper]{article}
\usepackage{amssymb}
\usepackage{amsmath}
\usepackage[top=1.2in, bottom=1.4in, left=1in, right=1in]{geometry}
\usepackage{parskip}
 \usepackage[T1]{fontenc}
\usepackage{lmodern}
 \usepackage{centernot}
\usepackage{epsfig}
\usepackage{xcolor}
\usepackage{amsmath,amsthm,amsfonts,amssymb,cite,amscd}
\usepackage{indentfirst}
\usepackage{mathrsfs}
\usepackage{amsthm}
\usepackage{tikz}
\usepackage{mathtools}
\usepackage{calc}
\usepackage[symbol]{footmisc}
\usetikzlibrary{patterns,arrows,decorations.pathreplacing}
\usetikzlibrary{fadings}
\usepackage{pgfplots}
\usepackage{parskip}
\newtheorem{theorem}{Theorem}

\newtheorem{lemma}[theorem]{Lemma}

\newtheorem{definition}{Definition}

\newtheorem{conjecture}{Conjecture}

\newcommand{\RNum}[1]{\uppercase\expandafter{\romannumeral #1\relax}}

\begin{document}
\baselineskip=0.30in

\begin{center}
{\LARGE  \bf \textbf{The number of triangles is more when they have no common vertex}}\\
\vskip 0.3cm
\author\centerline{Chuanqi Xiao \footnote[1]{email:chuanqixm@gmail.com}\\
Central European University, Budapest, Hungary\\
 Gyula O.H. Katona
\footnote[2]{email:katona.gyula.oh@renyi.hu}\\

 MTA R\'{e}nyi Institute, Budapest, Hungary\\}
  \end{center}
  {\bf Abstract}\ \
 By the theorem of Mantel \cite{MAN} it is known that a graph with $n$ vertices and $\lfloor \frac{n^{2}}{4} \rfloor+1$ edges must contain a triangle. A theorem of Erd\H os gives a strengthening: there are not only one, but at least $\lfloor\frac{n}{2}\rfloor$ triangles. We give a further improvement: if there is no vertex contained by all triangles then there are at least $n-2$ of them. There are some natural generalizations when $(a)$ complete graphs are considered (rather than triangles), $(b)$ the graph has $t$ extra edges (not only one) or $(c)$ it is supposed that there are no $s$ vertices such that every triangle contains one of them. We were not able to prove these generalizations, they are posed as conjectures.
 
 \section{Introduction}
 All graphs considered in this paper are finite and simple. Let $G$ be such a graph, the vertex set of $G$ is denoted by $V(G)$, the edge set of $G$ by $E(G)$, the number of vertices in $G$ is $v(G)$ and the number of edges in $G$ is $e(G)$. We denote the degree of a vertex $v$ by $d(v)$, the neighborhood of $v$ by $N(v)$, the number of edges between vertex sets $A$ and $B$ by $e(A,B)$ and the number of triangles in $G$ by $T(G)$. A $triangle~covering~set$ in $V(G)$ is a vertex set that contains at least one vertex of every triangle in $G$. The $triangle~covering~number$, denoted by $\tau_{\triangle}(G)$, is the size of the smallest triangle covering set. Let $S\subset V(G)$ be any subset of $V(G)$, then $G[S]$ is the subgraph induced by $S$.

 Mantel \cite{MAN} proved that an $n$-vertex graph with $\left\lfloor\frac{n^{2}}{4}\right\rfloor+t$ $(t\geq1)$ edges must contain a triangle. In 1941, Rademacher (unpublished, see \cite{E1}) showed that for even $n$, every graph $G$ on $n$ vertices and $\frac{n^2}{4}+1$ edges contains at least $\frac{n}{2}$ triangles and $\frac{n}{2}$ is the best possible. Later on, the problem was revived by Erd\H os, see \cite{E1}, which is now known as the Erd\H os-Rademacher problem, Erd\H os simplified Rademacher's proof and proved more generally that for $t\leq3$ and $n>2t$ case. Seven years later, he \cite{E2}
 conjectured that a graph with $\left\lfloor\frac{n^2}{4}\right\rfloor+t$ edges contains at least $t\left\lfloor\frac{n}{2}\right\rfloor$ triangles if $t<\frac{n}{2}$, which was proved by Lov\'asz and Simonovits \cite{LS}.
 Motivated by earlier results, we give a further improvement for the case $t=1$: if there is no vertex contained by all triangles then there are at least $n-2$ of them in $G$.

 \begin{theorem}[\textbf{Mantel}\cite{MAN}]\label{MAN} The maximum number of edges in an $n$-vertex triangle-free graph is $\lfloor \frac{n^{2}}{4}\rfloor$. Furthermore, the only triangle-free graph with $\lfloor \frac{n^2}{4}\rfloor$ edges is the complete bipartite graph $K_{\lfloor \frac{n}{2}\rfloor, \lceil\frac{n}{2} \rceil}$.
\end{theorem}

\begin{theorem}[\textbf{Erd\H os} \cite{E1}]\label{E1} Let $G$ be a graph with $n$ vertices and $\left\lfloor\frac{n^2}{4}\right\rfloor+t$ edges,  $t\leq3$, $n>2t$, then every $G$ contains at least $t\left\lfloor\frac{n}{2}\right\rfloor$ triangles.
\end{theorem} 

Before presenting our main result, the following definitions, a theorem and a lemma are needed.
\begin{definition}
Let $K_{i,n-i}$ denote a the complete bipartite graph on the vertex classes $|X|=i$, $|Y|=n-i$.
\end{definition}

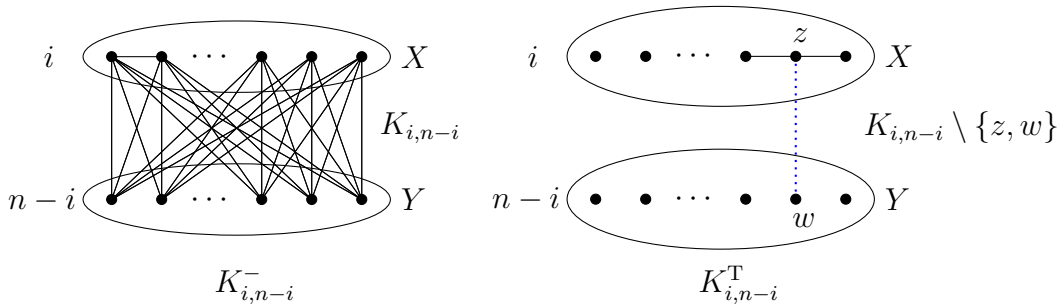
\begin{figure}[ht]
\centering
\begin{tikzpicture}[scale=.95]

\foreach \x in {-.7,-1.4,.7,1.4,2.1}{
\foreach \y in {-.7,-1.4,.7,1.4,2.1}{

\filldraw (\y,0)  -- (\x,-2) ;

}
}

\draw (0,0) node{$\cdots$} (0,-2) node{$\cdots$};


\draw (-.7,0) -- (-1.4,0)  ;


\draw  (2.5,-2) arc (0:360:2.15cm and .5cm) node[align=center,right]{$Y$};

\draw (-2.5,0) node[right]{$i$};

\draw (-3,-2) node[right]{$n-i$};

\draw (2.2,-1) node[right]{$K_{i,n-i}$};

\draw (2.5,0) arc (0:360:2.15cm and .5cm) node[align=center,right]{$X$};

\draw (0.6,-3.2) node{$K^{-}_{i,n-i}$};

\foreach \x in {-.7,-1.4,.7,1.4,2.1}{

\filldraw (\x,0) circle (2pt)  (\x,-2) circle (2pt);

}

\end{tikzpicture}
\begin{tikzpicture}[scale=.95]

\draw (0,0) node{$\cdots$} (0,-2) node{$\cdots$};


\draw (.7,0) -- (1.4,0) -- (2.1,0) ;
\draw [blue, thick, dotted] (1.4,0)--(1.4,-2);


\draw  (2.5,-2) arc (0:360:2.15cm and .7cm) node[align=center,right]{$Y$};

\draw (1.2,0.3) node[right]{$z$};
\draw (1.2,-2.3) node[right]{$w$};
\draw (-2.5,0) node[right]{$i$};

\draw (-3,-2) node[right]{$n-i$};

\draw (2.2,-1) node[right]{$K_{i,n-i}\setminus \{z,w\}$};

\draw (2.5,0) arc (0:360:2.15cm and .7cm) node[align=center,right]{$X$};

\draw (0.6,-3.2) node{$K^{\mathrm{T}}_{i,n-i}$};

\foreach \x in {-.7,-1.4,.7,1.4,2.1}{

\filldraw (\x,0) circle (2pt)  (\x,-2) circle (2pt);

}

\end{tikzpicture}
 \caption{Graphs $K^{-}_{i,n-i}$ and $K^{\mathrm{T}}_{i,n-i}$}
 \label{K_n}
\end{figure}
\begin{definition}
Let $K_{i,n-i}^{-}$ denote a graph obtained from a complete bipartite graph $K_{i,n-i}$ plus an edge in the class $X$ with $i$ vertices, see Figure \ref{K_n}. 
\end{definition}
\begin{definition}
Let $K_{i,n-i}^{\mathrm{T}}$ denote a graph obtained from a complete bipartite graph $K_{i,n-i}$ minus an edge plus two adjacent edges in the class $X$ with $i$ vertices, one end point of the missing edge is the shared vertex of these two adjacent edges and the other one is in the class $Y$, see Figure \ref{K_n}. 
\end{definition}
\begin{lemma}\label{L1}
Let $G$ be a graph with $n$ vertices and $\left\lfloor\frac{n^2}{4}\right\rfloor+1$ edges, such that $\tau_{\triangle}(G)=1$ and $T(G)\leq n-3$. Then $G$ is one of the following graphs: $K_{\frac{n}{2},\frac{n}{2}}^{-}$, $K_{\frac{n-1}{2},\frac{n+1}{2}}^{-}$, $K_{\frac{n+1}{2},\frac{n-1}{2}}^{-}$ or $K_{\frac{n+1}{2}, \frac{n-1}{2}}^{\mathrm{T}}$.
\end{lemma}

\begin{theorem}\label{T_{n}} Let $G$ be a graph with $n$ vertices and $\left\lfloor\frac{n^2}{4}\right\rfloor+1$ edges, then either $\tau_{\triangle}(G)=1$ or $T(G)\geq n-2$.
\end{theorem}
\section{Proofs of the main results}
\begin{proof}[Proof of Lemma \ref{L1}]
Let $v_{0}$ be such a vertex that $G\setminus v_0$ contains no triangle. We distinguish two cases.

\textbf{Case 1}.\ \ $G\setminus v_0$ contains at least one odd cycle. Let $C_{2k+1}$ ($k\geq 2$) be the shortest odd cycle in $G\setminus v_0$ and $G^{'}$ be the graph obtained from $G$ by removing the vertices of $C_{2k+1}$ and $v_{0}$, so $v(G^{'})=n-2k-2$. Since $C_{2k+1}$ is the shortest cycle in $G\setminus v_{0}$, each vertex in $G^{'}$ can be adjacent to at most $2$ vertices in the $C_{2k+1}$, otherwise, we can find a shorter odd cycle. Since $G^{'}$ is an $(n-2k-2)$-vertex triangle-free graph, by Theorem \ref{MAN}, $e(G^{'})\leq \left\lfloor\bigg(\frac{n-2k-2}{2}\bigg)^2\right\rfloor$. Obviously, any two vertices of $C_{2k+1}$ are not adjacent, therefore
\begin{align*} 
e(G\setminus v_{0})&\leq 2k+1+2(n-2k-2)+\left\lfloor\bigg(\frac{n-2k-2}{2}\bigg)^2\right\rfloor\\
&=k^{2}-nk+\left\lfloor \frac{n^2}{4}\right\rfloor+n-2\\
&\leq \left\lfloor\frac{n^2}{4}\right\rfloor-n+2~(k\geq2).
\end{align*}
Since $e(G)=d(v_{0})+e(G^{'})\leq (n-1)+( \left\lfloor\frac{n^2}{4}\right\rfloor-n+2)=\left\lfloor\frac{n^2}{4}\right\rfloor+1$, the only possibility for $e(G)=\left\lfloor\frac{n^2}{4}\right\rfloor+1$ is that $d(v_{0})=n-1$ and $e(G\setminus v_{0})= \left\lfloor\frac{n^2}{4}\right\rfloor-n+2$. In this case, we get $T(G)=\left\lfloor\frac{n^2}{4}\right\rfloor-n+2$, which contradicts $T(G)\leq n-3$.

\textbf{Case 2}.\ \ $G\setminus v_{o}$ has no odd cycles, then $G\setminus v_{0}$ is a bipartite graph and $e(G\setminus v_{0})\leq \left\lfloor\frac{n-1}{2}\right\rfloor\left\lceil\frac{n-1}{2}\right\rceil$. There are two subcases.

\textbf{Case 2.1}.\ \ $e(G\setminus v_{0})= \left\lfloor\frac{n-1}{2}\right\rfloor\left\lceil\frac{n-1}{2}\right\rceil$. Then $G\setminus v_{0}$ is $K_{\left\lfloor\frac{n-1}{2}\right\rfloor,\left\lceil\frac{n-1}{2}\right\rceil}$ and $d(v_{0})=e(G)-e(G\setminus v_{0})=\left\lfloor\frac{n}{2}\right\rfloor$+1. Let $d_1$ and $d_{2}$ be the numbers of neighbors of $v_{0}$ in classes $X$ and $Y$ of $K_{\left\lfloor\frac{n-1}{2}\right\rfloor,\left\lceil\frac{n-1}{2}\right\rceil}$, respectively, then $d(v_{0})=d_{1}+d_{2}$ and $T(G)=d_1d_2$. So we need $d_{1}+d_{2}=\left\lfloor\frac{n}{2}\right\rfloor+1$ and $d_{1}d_{2}\leq n-3$ hold true at the same time. When $n$ is even, we can see that the only solution is when $d_{1}=1$ and $d_{2}=\frac{n}{2}$. The symmetric solution, $d_{1}=\frac{n}{2}$, $d_{2}=1$ is not possible, since $d_{1}\leq \frac{n}{2}-1$ in this case. Therefore, we get that $G$ is $K_{\frac{n}{2},\frac{n}{2}}^{-}$. Assume now that $n$ is odd, there are two possibilities,

$(i)$ $d_{1}=1$ and $d_{2}=\frac{n-1}{2}$, in the same way as in the even case, we get $T(G)=\frac{n-1}{2}$ and $G$ is $K_{\frac{n+1}{2},\frac{n-1}{2}}^{-}$. When $d_{1}=\frac{n-1}{2}$ and $d_{2}=1$, we also get $T(G)=\frac{n-1}{2}$ and $G$ is $K_{\frac{n+1}{2},\frac{n-1}{2}}^{-}$.

$(ii)$ $d_{1}=2$ and $d_{2}=\frac{n-3}{2}$, then $T(G)=2(\frac{n-3}{2})=n-3$ and $G$ is $K_{\frac{n+1}{2},\frac{n-1}{2}}^{\mathrm{T}}$. Similarly, when $d_{1}=\frac{n-3}{2}$ and $d_{2}=2$, we get the same result.

\textbf{Case 2.2 }.\ \ $e(G\setminus v_{0})= \left\lfloor\frac{n-1}{2}\right\rfloor\left\lceil\frac{n-1}{2}\right\rceil-t$. Then $d(v_{0})=\left\lfloor\frac{n}{2}\right\rfloor$+1+t, $1\leq t\leq \left\lceil\frac{n}{2}\right\rceil-2$. Let $G\setminus v_{0}$ be the bipartite graph with partitions $X^{'}$ and $Y^{'}$, where $|X^{'}|=i^{'}$, then we have
$$i^{'}(n-1-i^{'})\geq \left\lfloor\frac{n-1}{2}\right\rfloor\left\lceil\frac{n-1}{2}\right\rceil-t$$
\begin{eqnarray}\label{i}
\Rightarrow \left\{
\begin{aligned}
& \frac{n-1-\sqrt{4t+1}}{2}\leq i^{'}\leq \frac{n-1+\sqrt{4t+1}}{2}, \text{n is even}, \\
& \frac{n-1-2\sqrt{t}}{2}\leq i^{'}\leq \frac{n-1+2\sqrt{t}}{2},  ~\text{n is odd}.
\end{aligned}
\right.\end{eqnarray}

Suppose $v_{0}$ has $d_{1}$ $(\geq 1)$ neighbors in $X^{'}$ and $d_{2}$ $(\geq 1)$ neighbors in $Y^{'}$. Since $G\setminus v_{0}$ is bipartite, if $d_1d_2=0$, then $G$ contains no triangle which contradicts the fact that  $\tau_{\triangle}(G)=1$. In this situation, $d_{1}d_{2}\geq T(G)\geq d_{1}d_{2}-t=d_{1}(\left\lfloor\frac{n}{2}\right\rfloor+1+t-d_{1})-t=-d_{1}^{2}+(\left\lfloor\frac{n}{2}\right\rfloor+1+t)d_{1}-t\geq -d_{1}^{2}+(\left\lfloor\frac{n}{2}\right\rfloor+1)d_{1}$.

When $n$ is even, we know that the solutions of $n-3\geq T(G)=d_1(\frac{n}{2}+1-d_1)$ is exactly one of $d_1=1$ or $d_2=1$ holds like in Case $2.1$. However, when $d_2=1$, since $d_1+d_2=\frac{n}{2}+1+t$, we have $d_1=\frac{n}{2}+t$, which contradicts $(1)$ namely $i'\leq \frac{n-1+\sqrt{4t+1}}{2}$ $(1\leq t\leq \frac{n}{2}-2)$ because $d_1\leq i^{'}$. The case $d_1=1$ and $d_2=\frac{n}{2}+t$ can be settled in the same way. 

When $n$ is odd, $n-3\geq T(G)=d_1(\left\lfloor\frac{n}{2}\right\rfloor+1-d_1)$ implies that one of $d_1=1$, $d_2=1$, $d_1=2$ or $d_2=2$ holds. By symmetry we can consider the cases $d_1=1$ and $d_1=2$. We check the details of the following $3$ subcases.

$(i)$ $t=1$ and $d_{1}=1$. We get $d_{2}=\frac{n+1}{2}$ because $d_1+d_2=\frac{n-1}{2}+1+t$. Since $d_2\leq|Y^{'}|=n-1-i'\leq\frac{n-1+2\sqrt{t}}{2}=\frac{n+1}{2}$, we get $|Y^{'}|=\frac{n+1}{2}$ and $|X^{'}|=\frac{n-3}{2}$. Since $e(G\setminus v_0)=\frac{n-1}{2}\frac{n-1}{2}-1$, we see that $G\setminus v_{0}$ is $K_{\frac{n-3}{2},\frac{n+1}{2}}$. Thus, $G$ is $K_{\frac{n-1}{2},\frac{n+1}{2}}^{-}$ and $T(G)\leq d_{1}d_{2}=\frac{n+1}{2}$.

$(ii)$ $t\geq2$ and $d_1=1$. By $d_1+d_2=\frac{n-1}{2}+1+t$, we have $d_{2}=\frac{n-1}{2}+t> \frac{n-1+2\sqrt{t}}{2}$, which contradicts  $d_2\leq|Y^{'}|=n-1-i'\leq\frac{n-1+2\sqrt{t}}{2}$.

$(iii)$ $t\geq1$ and $d_1=2$.  By $d_1+d_2=\frac{n-1}{2}+1+t$, we have $d_{2}=\frac{n-1}{2}+t-1$. However, $T(G)\geq d_1d_2-t=2(\frac{n-1}{2}+t-1)-t\geq n-2$, which contradicts $T(G)\leq n-3$. 

In conclusion, when $n$ is even, $G$ is $K_{\frac{n}{2},\frac{n}{2}}^{-}$. When $n$ is odd, $G$ is either $K_{\frac{n-1}{2},\frac{n+1}{2}}^{-}$ or $K_{\frac{n+1}{2},\frac{n-1}{2}}^{-}$ or $K_{\frac{n+1}{2},\frac{n-1}{2}}^{\mathrm{T}}$.
\end{proof}
Using Lemma \ref{L1}, we are able to give the proof of Theorem \ref{T_{n}}.
\begin{proof}[Proof of Theorem \ref{T_{n}}]
We prove our result by induction on $n$. The induction step will go from $n-2$ to $n$, so we check the bases when $n=3$ and $n=4$, obviously, our statement is true for these two cases. Suppose Theorem \ref{T_{n}} holds for $k= n-2$ $(n\geq5)$, we separate the rest of the proof into 2 cases.

\textbf{Case 1}.\ \ Every edge in $G$ is contained in at least one triangle. Then $T(G)\geq \left\lceil\frac{\left\lfloor\frac{n^2}{4}\right\rfloor+1}{3}\right\rceil\geq n-2$.

\textbf{Case 2}.\ \ There exists at least one edge $uv$ which is not contained in any triangle. Then $u$ and $v$ cannot have common neighbor in $G\setminus \{u,v\}$, which implies that $e(\{u,v\},G\setminus \{u,v\})\leq n-2$. Therefore, $e(G\setminus \{u,v\})\geq \left\lfloor\frac{n^2}{4}\right\rfloor-(n-2)=\left\lfloor\frac{(n-2)^2}{4}\right\rfloor+1$. In this point, we split the rest of the proof into $3$ subcases.

\textbf{Case 2.1}\ \ $e(G\setminus\{u,v\})\geq \left\lfloor\frac{(n-2)^2}{4}\right\rfloor+3$. By Theorem \ref{E1}, we get $T(G\setminus \{u,v\})\geq 3\left\lfloor \frac{n-2}{2}\right\rfloor$, which implies that $T(G)\geq 3\left\lfloor \frac{n-2}{2}\right\rfloor\geq n-2$.

\textbf{Case 2.2}.\ \ $e(G\setminus\{u,v\})= \left\lfloor\frac{(n-2)^2}{4}\right\rfloor+2$. When $n$ is even, by Theorem \ref{E1}, we get $T(G\setminus \{u,v\})\geq n-2$, since $T(G)\geq T(G\setminus \{u,v\})$, we are done. When $n$ is odd, we have $e(\{u,v\},G\setminus \{u,v\})= n-3$, then there exists $w\in V(G\setminus \{u,v\})$ such that edges $vw, uw\notin E(G)$. If $e(G[N(u)\setminus v])+e(G[N(v)\setminus u])\geq 1$, then the number of triangles which contains $u$ or $v$ is at least $1$. By Theorem \ref{E1}, $T(G\setminus \{u,v\})\geq n-3$ holds, thus, $T(G)\geq n-2$. Otherwise, $G\setminus \{u,v,w\}$ is bipartite and all triangles in $G\setminus \{u,v\}$ are adjacent to $w$. since $e(G[N(u)\setminus v])+e(G[N(v)\setminus u])=0$, no triangle contains $u$ or $v$. Therefore, $\tau_{\triangle}(G)=\tau_{\triangle}(G\setminus \{u,v\})=1$ and all triangles in $G$ are adjacent to $w$.

\textbf{Case 2.3}. \ $e(G\setminus \{u,v\})= \left\lfloor\frac{(n-2)^2}{4}\right\rfloor+1$, then $e(\{u,v\},G\setminus \{u,v\})=n-2$. When $e(G[N(u)\setminus v])+e(G[N(v)\setminus u])=0$,
$G\setminus \{u,v\}$ is bipartite, it has at most $\left\lfloor\frac{(n-2)^2}{4}\right\rfloor$ edges, contradicting the assumption of the case.

Suppose $e(G[N(u)\setminus v])+e(G[N(v)\setminus u])=1$. Since $|N(u)\setminus v~\cup~ N(v)\setminus u|=n-2$, we have $e([N(u)\setminus v],[N(v)\setminus u])\leq \left\lfloor\frac{(n-2)^2}{4}\right\rfloor$. Thus, $e(G\setminus \{u,v\})=\left\lfloor\frac{(n-2)^2}{4}\right\rfloor+1$ implies that $G\setminus \{u,v\}$ is obtained from $K_{\left\lfloor\frac{n-2}{2}\right\rfloor,\left\lceil\frac{n-2}{2}\right\rceil}$ plus an edge, say $\{j,k\}$, in one class. Therefore, all triangles in $G$ contain $\{j,k\}$ and hence $\tau_{\triangle}(G)=1$ follows.

Now we assume that $e(G[N(u)\setminus v])+e(G[N(v)\setminus u])\geq 2$, then the number of the triangles containing $u$ or $v$ is at least 2. It is easy to check that if $v(G)=5$ then $G\setminus \{u,v\}$ is a triangle and either $\tau_{\triangle}(G)=1$ or $T(G)=4$. Therefore, we may assume $n\geq 6$. Since $e(G\setminus \{u,v\})=\left\lfloor\frac{(n-2)^2}{4}\right\rfloor+1$, by the induction hypothesis, either $\tau_{\triangle}(G\setminus \{u,v\})=1$ or $T(G\setminus \{u,v\})\geq n-4$. When $T(G\setminus \{u,v\})\geq n-4$, we have $T(G)\geq T(G\setminus \{u,v\})+2\geq n-2$. Otherwise, $\tau_{\triangle}(G\setminus \{u,v\})=1$ and $T(G\setminus \{u,v\})\leq n-5$ hold.
By Lemma \ref{L1}, we see that when $n$ is even, 
$G\setminus\{u,v\}$ is $K_{\frac{n}{2}-1,\frac{n}{2}-1}^{-}$, when $n$ is odd, $G\setminus\{u,v\}$ is either $K_{\frac{n-3}{2},\frac{n-1}{2}}^{-}$ or $K_{\frac{n-1}{2},\frac{n-3}{2}}^{-}$ or $K_{\frac{n-1}{2},\frac{n-3}{2}}^{\mathrm{T}}$. Let us check what will happen in these cases.

\begin{figure}[ht]
\centering
\begin{tikzpicture}[scale=.95]

\filldraw (-2.2,0.5) circle (2 pt) (-2.2,-2) circle (2 pt);
\draw (0.8,0) node{$\cdots$} (.8,-2) node{$\cdots$};


\draw (-.6,0) -- (-1.2,0) (-2.2,0.5)--(-2.2,-2);
\draw[red](-.6,0)--(-2.2,0.5)--(-1.2,0) (-2.2,0.5)--(-1.2,-2);
\draw ;


\draw  (2.5,-2) arc (0:360:2.15cm and .7cm) node[align=center,right]{$Y$};

\draw (-1.4,0.1) node[below]{$j$};
\draw (-.7,0) node[below]{$k$};
\draw (-1.4,-2.3) node[right]{$l$};
\draw (-2.5,0.8) node[right]{$u$};

\draw (-2.5,-2.3) node[right]{$v$};

\draw (2.5,0) arc (0:360:2.15cm and .7cm) node[align=center,right]{$X$};

\draw (0.6,-3.4) node{$K_{|X|,|Y|}^{-}$};

\foreach \x in {-.6,-1.2,0,1.3,2.1}{

\filldraw (\x,0) circle (2pt)  (\x,-2) circle (2pt);

}

\end{tikzpicture}\qquad\qquad
\begin{tikzpicture}[scale=.95]

\filldraw (-2.2,0.5) circle (2 pt) (-2.2,-2) circle (2 pt);
\draw (0.8,0) node{$\cdots$} (.8,-2) node{$\cdots$};


\draw (0,0)--(-.6,0) -- (-1.2,0) (-2.2,0.5)--(-2.2,-2);
\draw[red](-.6,0)--(-2.2,0.5);
\draw ;


\draw  (2.5,-2) arc (0:360:2.15cm and .7cm) node[align=center,right]{$Y$};

\draw (-1.4,0.1) node[below]{$j$};
\draw (-.7,0) node[below]{$z$};
\draw (0,0) node[below]{$k$};
\draw (-.7,-2.3) node[right]{$w$};
\draw (-2.5,0.8) node[right]{$u$};

\draw (-2.5,-2.3) node[right]{$v$};
\draw (2.2,-1) node[right]{$K_{\frac{n-1}{2},\frac{n-3}{2}}\setminus \{z,w\}$};
\draw [blue, thick, dotted] (-.6,0)--(-.6,-2);

\draw (2.5,0) arc (0:360:2.15cm and .7cm) node[align=center,right]{$X$};

\draw (0.6,-3.4) node{$K_{\frac{n-1}{2},\frac{n-3}{2}}^{\mathrm{T}}$};

\foreach \x in {-.6,-1.2,0,1.3,2.1}{

\filldraw (\x,0) circle (2pt)  (\x,-2) circle (2pt);

}

\end{tikzpicture}
 \caption{}
 \label{2}
\end{figure}
We first give the following technical lemma:
\begin{lemma}\label{com}
Let $f(a,b)=ab+(A-a)(B-b)$, where $A$ and $B$ are integers, $1\leq a\leq A$, $1\leq b\leq B$, then $f(a,b)\geq min\{A,B\}$.
\end{lemma}
\begin{proof}[Proof of Lemma \ref{com}]
Obviously, when $AB=max\{A,B\}$, $f(a,b)\geq 1=min\{A,B\}$. Otherwise, we have $A,B\geq 2$.
Without loss of generality, fix $b$, then $f(a,b)$ is a linear function of variable $a$. Since $\frac{\partial f}{\partial a}=b-(B-b)$, thus, $f(a,b)$ is decreasing when $b<\frac{B}{2}$ and $f(a,b)$ is increasing when $b>\frac{B}{2}$. Therefore,
\begin{equation*}
    f(a,b)\geq
    \begin{cases}f(A,b)=Ab,~~b\leq \frac{B}{2},\\
    f(1,b)=b+(A-1)(B-b),~~b>\frac{B}{2}.
    \end{cases}
\end{equation*}
It is easy to check that $Ab\geq A$, when $b\leq \frac{B}{2}$, and $b+(A-1)(B-b)=B(A-1)+b(2-A)\geq B $ when $b>\frac{B}{2}$. Hence, we get $f(a,b)\geq min\{A,B\}$. Obviously, if $min\{A,B\}=A$, the equality holds only when $a=A$ and $b=1$, if $min\{A,B\}=B$, the equality holds only when $a=1$ and $b=B$.
\end{proof}

\textbf{Case 2.3.1}. \ $G\setminus\{u,v\}$ is $K^{-}_{\left\lfloor\frac{n}{2}\right\rfloor-1,\left\lceil\frac{n}{2}\right\rceil-1}$, which implies that when $n$ is even, $G\setminus\{u,v\}$ is $K^{-}_{\frac{n}{2}-1,\frac{n}{2}-1}$ and when $n$ is odd, $G\setminus\{u,v\}$ is $K^{-}_{\frac{n-3}{2},\frac{n-1}{2}}$. Let $X$ and $Y$ be the two classes of $K^{-}_{\left\lfloor\frac{n}{2}\right\rfloor-1,\left\lceil\frac{n}{2}\right\rceil-1}$ and $\{j,k\}$ be the extra edge in $X$, where $|X|=\left\lfloor\frac{n}{2}\right\rfloor-1$, see Figure $2$. Since $e(G[N(u)\setminus v])+e(G[N(v)\setminus u])\geq 2$, $\mid N(u)\setminus v \cup N(v)\setminus u\mid=n-2$ and $N(u)\setminus v\cap N(v)\setminus u=\emptyset$, we see that either $N(u)\setminus v$ or $N(v)\setminus u$ contains at least one vertex in both classes $X$ and $Y$. Without loss of generality, say at least $N(u)\setminus v$ has this property.

Let $|N(u)\setminus v~\cap X|=a$ and $|N(u)\setminus v~\cap Y|=b$, where $1\leq a\leq \left\lfloor\frac{n}{2}\right\rfloor-1$ and $1\leq b\leq \left\lceil\frac{n}{2}\right\rceil-1$. Then the number of triangles which are adjacent to $u$, containing one vertex in $X$ and one in $Y$ is $ab$ while the number of triangles which are adjacent to $v$, containing one vertex in $X$ and one in $Y$ is $(A-a)(B-b)$. Hence, we get $T(G)\geq ab+\bigg(\left\lfloor\frac{n}{2}\right\rfloor-1-a\bigg)\bigg(\left\lceil\frac{n}{2}\right\rceil-1-b\bigg)+\left\lceil\frac{n}{2}\right\rceil-1$. By Lemma \ref{com}, we see $T(G)\geq \left\lfloor\frac{n}{2}\right\rfloor-1+\left\lceil\frac{n}{2}\right\rceil-1=n-2$.

\textbf{Case 2.3.2}. $n$ is odd and $G\setminus\{u,v\}$ is $K_{\frac{n-1}{2},\frac{n-3}{2}}^{-}$.  Let $X$ and $Y$ be the two classes of $K^{-}_{\frac{n-1}{2},\frac{n-3}{2}}$ and $\{j,k\}$ be the extra edge in $X$, where $|X|=\frac{n-1}{2}$. Similarly as in the previous case, either $N(u)\setminus v$ or $N(v)\setminus u$ contains at least one vertex in both classes $X$ and $Y$. Without loss of generality, say at least $N(u)\setminus v$ has this property.

Let $|N(u)\setminus v~\cap X|=a$ and $|N(u)\setminus v~\cap Y|=b$, where $1\leq a\leq \frac{n-1}{2}$ and $1\leq b\leq \frac{n-3}{2}$, then $T(G)\geq ab+\bigg(\frac{n-1}{2}-a\bigg)\bigg(\frac{n-3}{2}-b\bigg)+\frac{n-3}{2}$. By Lemma \ref{com}, we get $T(G)\geq \frac{n-3}{2}+\frac{n-3}{2}\geq n-3$, the equality holds only if $a=1$ and $b=\frac{n-3}{2}$. Let $s\in X$ and $\{u,s\}\in E(G)$, $a=1$ and $b=\frac{n-3}{2}$ implies that either $s\in \{j,k\}$ then $\tau_{\triangle}(G)=1$, or $s\notin\{j,k\}$ then there exists one more triangle $\{v,j,k\}$, thus $T(G)\geq n-3+1=n-2$.

\textbf{Case 2.3.3}. $n$ is odd and $G\setminus\{u,v\}$ is $K_{\frac{n-1}{2},\frac{n-3}{2}}^{\mathrm{T}}$. Since $\frac{n-1}{2}\geq 3$, we get $n\geq 7$. Let $X$ and $Y$ be the classes of $K_{\frac{n-1}{2},\frac{n-3}{2}}^{\mathrm{T}}$, $\{j,z\}$ and $\{z,k\}$ be the two extra edges in $X$ and $\{z,w\}$ be the missing edge in $K_{\frac{n-1}{2},\frac{n-3}{2}}$, see Figure \ref{2}.

Let $|N(u)\setminus v~\cap X|=a$ and $|N(u)\setminus v~\cap Y|=b$. Since  $\mid N(u)\setminus v \cup N(v)\setminus u\mid=n-2$ and $N(u)\setminus v\cap N(v)\setminus u=\emptyset$, when $a=0$, we have $X\subseteq N(v)\setminus u$. If $N(v)\setminus u=X$, clearly, all triangles in $G$ contain $z$ and hence $\tau_{\triangle}(G)=1$. Otherwise, $|(N(v)\setminus u)\cap Y|\geq 1$. It is easy to check that $T( K_{\frac{n-1}{2},\frac{n-3}{2}}^{\mathrm{T}})=n-5$, therefore, in this case we get $T(G)\geq n-5+2+\frac{n-1}{2}-1\geq n-1$ $(n\geq 7)$.
When $b=0$, then $Y\subseteq N(v)\setminus u$. If $N(v)\setminus u=Y$ then $N(u)\setminus v=X$, we see that all triangles in $G$ contain $z$ and hence $\tau_{\triangle}(G)=1$. Otherwise, $|(N(v)\setminus u)\cap X|\geq 1$. When $|(N(v)\setminus u)\cap X|=1$, if $(N(v)\setminus u)\cap X=\{z\} $, obviously, all triangles in $G$ contain $z$, hence $\tau_{\triangle}(G)=1$. If not, then clearly $T(G)\geq n-5+1+\frac{n-3}{2}\geq n-2$ $(n\geq 7)$. It is easy to check that $T(G)$ reaches the lower bound when $|(N(v)\setminus u)\cap X|=1$ for $n\geq 9$ and when $n=7$, $T(G)\geq 5$ holds in all cases. Therefore, we get either $\tau_{\triangle}(G)=1$ or $T(G)\geq n-2$.

 Now suppose that, $1\leq a\leq \frac{n-1}{2}$ and $1\leq b\leq \frac{n-3}{2}$. Then $T(G)\geq ab+(\frac{n-1}{2}-a)(\frac{n-3}{2}-b)+n-5$, by Lemma \ref{com}, we get $T(G)\geq \frac{n-3}{2}+n-5\geq n-2~(n\geq 9)$. Since $T(G)\geq 5$ when $n=7$, we see that $T(G)\geq n-2$ holds in this case.

This completes the proof.
\end{proof}

\section{Open problems}
Let $V_1, V_2, \ldots , V_r$ be pairwise disjoint sets where $ \left\lceil \frac{n}{2}\right\rceil \geq |V_1|\geq |V_2|\geq \ldots \geq |V_r|\geq \left\lfloor\frac{n}{2}\right\rfloor$  and $\sum |V_i|=n$ hold. Define the graph $T_r(n)$ with vertex set $\cup V_i$ where $\{ u,v\}$ is an edge
if $u\in V_i$, $v\in V_j (i\not=j)$, but there is no edge within a $V_i$. The number of edges of the graph $T_r(n)$ is denoted by $t_r(n)$. The following fundamental theorem of Tur\'an is a generalization of Mantel's theorem.

\begin{theorem} [\textbf{Tur\'an}\cite{TUR}]\label{TUR}
 If a graph on $n$ vertices has more than $t_{k-1}(n)$ edges then it contains a copy of the complete graph $K_k$ as a subgraph.
\end{theorem}

The most natural construction is to add one edge to $T_{k-1}(n)$ in the set $V_1$. This graph is denoted by
$T_{k-1}^{\--}(n)$. It contains not only one copy of $K_k$ but $|V_2|\cdot |V_3|\cdots |V_{k-1}|$ of them.
\cite{E3} proved that this is the least number. Observe that the intersection of all of these copies of $K_k$
is a pair of vertices (in $V_1$). If this is excluded, the number of copies probably increases. This is expressed
by the following conjecture. Take $T_{k-1}(n)$, add an edge $\{ x,y\}$ in $V_1$, an edge $\{ u,v\}$ in $V_2$ and delete the edge $\{u,x\}$. This graph is denoted by $T_{k-1}^{\sqsubset}$. It contains almost the double of the number of copies of $K_k$ in $T_{k-1}^{\--}(n)$.

\begin{conjecture} If a graph on $n$ vertices has $t_{k-1}(n)+1$ edges and the copies of $K_k$ have an empty intersection then the number of copies of $K_k$ is at least as many as in   $T_{k-1}^{\sqsubset}$: $\ (|V_2|-1) |V_3|\cdot |V_4|\cdots |V_{k-1}|+(|V_1|-1) |V_3|\cdot |V_4|\cdots |V_{k-1}|=(|V_1|+|V_2|-2)|V_3|\cdot |V_4|\cdots |V_{k-1}|$.
\end{conjecture}

Of course this would be a generalization of our Theorem 4. Now we try to generalize it in a different direction.
What is the minimum number of triangles in an $n$-vertex graph $G$ containing $\left\lfloor \frac{n^{2}}{4}\right\rfloor+t$ edges if $\tau_{\triangle}(G)\geq s$ is also supposed. The problem is interesting only when $0<t<s$. Otherwise, if
$t\geq s$ then $\tau_{\triangle}(G)=t$ is allowed. By Lov\'asz-Simonovits' theorem \cite{LS}, we know that the number of triangles is
at least $t\left\lfloor \frac{n}{2}\right\rfloor$ with equality for the following graph. Take $K_{\left\lceil \frac{n}{2}\right\rceil, \left\lfloor \frac{n}{2}\right\rfloor}$ where the two parts are $V_1 (|V_1|=\left\lceil \frac{n}{2}\right\rceil)$ and $V_2 (|V_2|=\left\lfloor\frac{n}{2}\right\rfloor)$, respectively. Add $t$ edges to $V_1$. Here all triangles contain one of the new added edges, therefore $\tau_{\triangle}(G)\leq t$ and the extra condition on $\tau_{\triangle}(G)$ is not a real restriction.

 Hence we may suppose $0<t<s$. Choose $2(s-1)$ distinct vertices in $V_1$ (of $K_{\left\lceil \frac{n}{2}\right\rceil, \left\lfloor\frac{n}{2}\right\rfloor}$): $x_1, x_2, \ldots , x_{s-1}, y_1, y_2, \ldots , y_{s-1}$ and two distinct vertices in $V_2: u_1, u_2$. Add the edges $\{x_1, y_1\}$, $\{x_2, y_2\}, \ldots , \{x_{s-1}, y_{s-1}\}, \{u_1, u_2\}$ to $K_{\left\lceil \frac{n}{2}\right\rceil, \left\lfloor\frac{n}{2}\right\rfloor}$ and delete the edges $\{ x_1 ,u_1\}, \ldots, \{ x_{s-t}, u_1\}$. Let $K_{\left\lceil \frac{n}{2}\right\rceil, \left\lfloor \frac{n}{2}\right\rfloor}^{s,t}$ denote this graph. It is easy to see that it contains $\left\lfloor\frac{n^{2}}{4} \right\rfloor+t$ edges. On the other hand it contains $s$
vertex disjoint triangles if $\left\lceil \frac{n}{2}\right\rceil \geq 2(s-1)+1$ and $ \left\lfloor \frac{n}{2}\right\rfloor\geq s+1$. Therefore, $\tau_{\triangle}(K_{\left\lceil \frac{n}{2}\right\rceil, \left\lfloor \frac{n}{2}\right\rfloor}^{s,t})=s$ holds if $n$ is large enough. We believe that this is the best possible construction.

\begin{conjecture}
 Suppose that the graph $G$ has $n$ vertices and $\left\lfloor \frac{n^{2}}{4} \right\rfloor+t$ edges, it satisfies
 $\tau_{\triangle}(G)\geq s$ and $n\geq n(t,s)$ is large. Then $G$ contains at least as many triangles as
$K_{\left\lceil\frac{n}{2}\right\rceil, \left\lfloor\frac{n}{2}\right\rfloor}^{s,t}$ has, namely
$(s-1)\left\lfloor \frac{n}{2}\right\rfloor+\left\lceil\frac{n}{2}\right\rceil -2(s-t)$.
\end{conjecture}
 In the case $t=1, s=2$ our Theorem 4 is obtained. There is an obvious common generalization of our two conjectures.

\vskip 10mm
{\Large \noindent \bf Acknowledgments}\ \

We thank Jimeng Xiao for his suggestions to improve Lemma \ref{L1}.

\end{document}